\documentclass[11pt]{article}
\usepackage[utf8]{inputenc}
\usepackage{amsfonts,amsmath,amssymb,mathrsfs,enumerate,amsthm}
\usepackage[textwidth=6.9in, textheight=9.4in]{geometry}

\usepackage{color}
\usepackage{graphicx}
\usepackage{booktabs}
\usepackage{xcolor}
\usepackage{pstricks, pst-node, pst-plot, pst-math, tikz}
\usepackage[hidelinks]{hyperref}
\usetikzlibrary{arrows,quotes,angles,cd}
\usepackage{multicol}
\usepackage{framed, color}
\usepackage[shortlabels]{enumitem}
\newlist{todolist}{itemize}{2}
\setlist[todolist]{label=$\square$}

\definecolor{shadecolor}{rgb}{1.0, 1.0, 0.65}

\usepackage{pgfplots}
\pgfplotsset{
	humanaxes/.style={axis lines=center, every axis plot/.append style={very thick, mark size=3}, x axis line style=-, y axis line style=-},
	humanaxeslabels/.style={every axis x label/.style={at={(current axis.right of origin)},anchor=west},every axis y label/.style={at={(current axis.above origin)},anchor=south}},
	human/.style={humanaxes, humanaxeslabels}
}
\pgfplotsset{my style/.append style={axis x line=middle, axis y line=
middle, xlabel={$x$}, ylabel={$y$}, axis equal }}

\usepgfplotslibrary{polar}
\usepgflibrary{shapes.geometric}
\usetikzlibrary{calc}
\usetikzlibrary{decorations.shapes}

\newcommand{\G}{\textbf{G}}
\newcommand{\St}{\textbf{S}}

\newcommand{\greg}{\ensuremath{\mathfrak{g}_{reg}}}
\newcommand{\fnum}{\ensuremath{F_i^{\#}}}

\usepackage[normalem]{ulem}

\DeclareMathOperator{\Id}{Id}
\DeclareMathOperator{\Sp}{Sp}
\pgfplotsset{soldot/.style={color=black,only marks,mark=*}}
\pgfplotsset{holdot/.style={color=black,fill=white,only marks,mark=*}}
\usepackage{microtype}
\pgfplotsset{compat=1.15}

\newtheorem{thm}{Theorem}[section]
  
  \newtheorem{lemma}[thm]{Lemma}
  
  \newtheorem{defn}[thm]{Definition}

  \makeatletter
\def\Ddots{\mathinner{\mkern1mu\raise\p@
\vbox{\kern7\p@\hbox{.}}\mkern2mu
\raise4\p@\hbox{.}\mkern2mu\raise7\p@\hbox{.}\mkern1mu}}
\makeatother

\usepackage{mathdots}

\usepackage{titling}
\usepackage{lipsum}

\title{Maximal Unramified Tori in Symplectic Groups}
\author{Jacob Haley}
\date{\today}

\begin{document}

\maketitle

\begin{abstract}
For a reductive group over a $p$-adic field, DeBacker in \cite{debacker} gives a paramaterization of the conjugacy classes of maximal unramified tori using Bruhat-Tits theory. On the other hand, for unramified classical groups, Waldspurger in \cite{wald} gives a parameterization in terms of triples of partitions by constructing a regular semisimple element whose structure is governed by the parts of the partitions. In this paper, we compare these two parameterizations in the case of the symplectic group Sp$_{2n}$.  
\end{abstract}

\section{Introduction}
If $\textbf{G}$ is a reductive algebraic group over a $p$-adic field $F$ with residue field $\mathbb{F}_q$, an unramified torus is the group of $F$-rational points of a $F$-torus in $\G$ that splits over an unramified extension of $F$. In \cite{debacker}, the author gives a parameterization of $\G(F)$-conjugacy classes of maximal unramified tori using Bruhat-Tits theory. In particular, modulo an equivalence, the conjugacy classes are parameterized by pairs $(\textsf{F},\textsf{T})$ where $\textsf{F}$ is a facet in the $F$-points of the Bruhat-Tits building of $\G$ and $\textsf{T}$ is an elliptic $\mathbb{F}_q$-torus in the reductive quotient of $\G$ at the facet $F$. Given such a pair, $\textsf{T}$ can be lifted to a maximal unramified torus in $\G$, and this gives a representative for the conjugacy class associated to the pair. The elliptic $\mathbb{F}_q$-tori in the reductive quotient are in turn parameterized by the elliptic Frobenius conjugacy classes in the Weyl group of the reductive quotient. 

On the other hand, for symplectic, orthogonal, and unramified unitary groups, Waldspurger in \cite{wald} gives a parameterization of the maximal unramified tori in terms of triples of partitions. For each part of one of these partitions, he defines an $F$-algebra whose structure is determined by the partition, and he also constructs an $F$-endomorphism of the algebra. Taking the sum of these $F$-algebras, he obtains a symplectic $F$-vector space, and the sum of the $F$-endomorphisms determines a regular semisimple element for a torus in the associated conjugacy class. 

The goal of this paper is to provide a comparison of the two parameterizations in the case of the symplectic group, similar to the work in \cite{nevins1} and \cite{nevins2} for nilpotent orbits. In particular, given a triple of partitions $(\mu_0, \mu', \mu'')$ in the Waldspurger parameterization, we will associate a facet in the Bruhat-Tits building and an elliptic conjugacy class in the Weyl group of the reductive quotient which determine the same conjugacy class of tori as $(\mu_0,\mu',\mu'')$. We will also construct an inverse map, showing that the two indexing sets are in bijective correspondence. Note that similar results exist for orthogonal and unramified unitary groups, and they will be included in future work of the author.

The facet corrresponding to a triple $(\mu_0,\mu',\mu'')$ will correspond to a subdiagram of the extended Dynkin diagram of type $C_n$. This subdiagram will be a union of two subdiagrams of type $C_k$, which will be determined by the partitions $\mu'$ and $\mu''$, and subdiagrams of type $A_j$, which will be determined by the partition $\mu_0$. 

Each element in the Weyl group of the reductive quotient attached to this facet is a product of signed $k_i$-cycles for natural numbers $k_i < n$. We say that a signed $k_i$-cycle $\sigma$ is even if $\sigma^{k_i} = \Id$ and odd if $\sigma^{k_i} = -\Id$. The conjugacy class of an element is determined by its cycle-type, and the parts of the partitions $\mu_0, \mu',$ and $\mu''$ give us the cycle-type of the associated element $w$. In particular, each part $x_i$ of $\mu_0$ give us an even $x_j$-cycle while each part $y_j$ of $\mu'$ or $\mu''$ give us an odd $y_j$-cycle.

The paper is organized as follows. In Section 2, we introduce some of the general notation needed for our result, and we recall some of the structure of the symplectic group $\Sp_{2n}$. In Section 3, we discuss the Bruhat-Tits building of $\Sp_{2n}$, and we discuss the DeBacker parameterization in more detail. In Section 4, we discuss Waldspurger's parameterization for $\Sp_{2n}$. In Section 5, we present and prove the main result, discussing each partition in Waldspurger's parameterization in a separate subsection before putting everything together. Finally, in Section 6 we discuss an inverse to our construction, showing how one can move from a pair $(\textsf{F},w)$ to the associated triple of partitions in the Waldspurger construction.

The author thanks Stephen DeBacker for suggesting the problem, consistently providing helpful advice, and also providing the below figures for an alcove in Sp$_4$.

\section{General Notation and the Group Sp$_{2n}$}

Let $F$ be a finite extension of $\mathbb{Q}_p$. Let $\mathfrak{o}$ be the ring of integers of $F$ and $\mathfrak{p}$ the maximal ideal in $\mathfrak{o}$. Let $\varpi$ be a uniformizer of $F$. Let $q$ be the order of the residue field of $F$, which we denote by $\mathbb{F}_q$. We let Fr denote a topological generator of the absolute Galois group of $\mathbb{F}_q$. We can and do identify Fr with a topological generator of Gal$(F_{\text{un}}/F)$, where $F_{\text{un}} \subseteq \overline{F}$ is the maximal unramified extension of $F$. 


If $\textbf{C}$ is an algebraic group defined over $F$, we will by abuse of notation also use $\textbf{C}$ (bold) to denote the $\overline{F}$-points, and we will use $C$ (not bold) to denote the group of $F$-points.

Let $V$ be a vector space over $F$ of even dimension $2n$ for $p \geq 6n + 1$,\footnote{This is Waldspurger's assumption on $p$ in \cite{wald}. 
In our proofs, we require $q > 2n$ when choosing the generators $a_i$ in our sections 5.1 - 5.3. We also want $p \neq 2$, else regular semisimple elements may not exist in the residue field.} and fix an ordered basis $B = \{e_1, \dots, e_n\}$ of $V$. Let $q_V$ be the non-degenerate anti-symmetric bilinear form of $V$ whose matrix with respect to $B$ is the anti-diagonal matrix

\begin{center}
$\begin{pmatrix}
& & & & & 1 \\
& & & & \iddots & \\
& & & 1 & & \\
& & -1 & & & \\
& \iddots & & & & \\
-1 & & & & & \\
\end{pmatrix}$
\end{center}

Let $\textbf{G}$ be the symplectic group $\Sp_{2n}$ preserving $q_V$. Then $\textbf{G}$ is a connected reductive group defined over $F$. We will use $\mathfrak{g}$ to denote the vector space of $F$-rational points of the Lie algebra of $\G{}$, which we may identify with a subalgebra of the $F$-endomorphisms of $V$.

If $\St$ denotes the diagonal torus in $\G{}$, which of course is an $F$-split maximal torus in $\G$, we will let diag($t_1, \dots, t_n$) denote the matrix in $\St$ whose $(i,i)$ entry is $t_i$ for $i \leq n$ and $t_{2n + 1 - i}^{-1}$ for $i > n$. Let $W = N_\G(\St)/\St$ be the Weyl group of $\St$ in $\G$, and we will let $\Phi$ denote the set of roots of $\St$ in $\G$. We fix the simple system $\Delta = \{\alpha_1, \dots, \alpha_{n-1}, \beta\}$, where $\alpha_j$ is the root that takes diag($t_1, \dots, t_n$) to $t_jt_{j+1}^{-1}$ for $1 \leq j \leq n-1$ and $\beta$ is the root which takes diag($t_1, \dots, t_n$) to $t_n^2$. Then $e := 2\alpha_1 + \dots + 2\alpha_{n-1} + \beta$ is the highest root with respect to this simple system, and it takes diag($t_1, \dots, t_n$) to $t_1^2$.

We conclude this section by recalling the root space decomposition of $\mathfrak{g}$. We let $E_{i,j}$ denote the elementary matrix having 1 in the $(i,j)$ entry and 0 elsewhere. Then we have

\begin{itemize}
    \item For $1 \leq i \neq j \leq n,$ the matrix $E_{i,j} - E_{2n + 1 - j, 2n + 1 - i}$ spans the root space of the root sending diag($t_1, \dots, t_n$) to $t_it_j^{-1}$.
    
    \item For $1 \leq i < j \leq n$, $E_{i,2n + 1 - j} + E_{j,2n + 1 - i}$ spans the root space of the root sending diag($t_1, \dots, t_n$) to $t_it_j$, while $E_{2n + 1 - j,i} + E_{2n + 1 - i,j}$ spans the root space of the root sending diag($t_1, \dots, t_n$) to $(t_it_j)^{-1}$.
    
    \item For $1 \leq i \leq n$, $E_{i,2n+1-i}$ spans the root space of the root sending diag($t_1, \dots, t_n$) to $t_i^2$, while $E_{2n+1-i,i}$ spans the root space of the root sending diag($t_1, \dots, t_n$) to $t_i^{-2}$.
\end{itemize}

Notice that if $X \in \mathfrak{g}$ has the block form $\begin{pmatrix} X^{1,1} & X^{1,2} \\ X^{2,1} & X^{2,2} \end{pmatrix}$, where each block is an $n \times n$ matrix, then the root spaces of the roots in the subsystem spanned by $\alpha_1, \dots, \alpha_{n-1}$ are contained in $X^{1,1}$ and $X^{2,2}$, while every root space lying in $X^{1,2}$ or $X^{2,1}$ must be associated to a root so that the coordinate of $\beta$ with respect to the simple system $\Delta$ is non-zero. 

\section{The DeBacker Parameterization for \textbf{G}}


\indent Let $\mathcal{B}(G) = \mathcal{B}(\G,F)$ be the Bruhat-Tits building of $G$, which we identify with the Fr-fixed points of $\mathcal{B}(\G,F_{\text{un}})$. We let $\mathcal{A} = \mathcal{A}(S)$ denote the apartment of the diagonal torus $S$ in $\mathcal{B}(G).$ Within the apartment, we fix a fundamental alcove $\mathcal{C}$ whose walls are determined by the hyperplanes of $n+1$ affine roots so that the set of their gradients is precisely $\Delta \cup \{ e \}$. By a slight abuse of notation, we denote these hyperplanes by $H_{\alpha_1}, \dots, H_{\alpha_{n-1}}, H_\beta, H_e.$

The facets in our fundamental alcove $\mathcal{C}$ can be identified with proper subdiagrams of the extended Dynkin diagram of $\G$. In particular, given such a subdiagram $\Gamma$, the corresponding facet is the one vanishing at the hyperplanes $H_\gamma$, where $\gamma \in \{ e, \alpha_1, \dots, \alpha_{n-1}, \beta \}$ is a root so that the corresponding vertex in the extended Dynkin diagram occurs in $\Gamma.$ We adopt the following notation for the facets in the fundamental alcove:

\begin{defn}
If $a, b \geq 0$ and $x_1, \dots, x_t > 0$ are integers so that $a + b + \sum_{i = 1}^{t}x_i = n$, the facet $\langle a | x_1, \dots, x_t | b \rangle$ is defined as follows. Let $$H_e^a = \begin{cases} \mathcal{A} & a = 0 \\ H_e & a \neq 0 \end{cases},$$ and define $H_\beta^b$ analogously. Then $\langle a | x_1, \dots, x_t | b \rangle$ is the facet in $\mathcal{C}$ lying on \begin{align*} H_e^a \cap H_{\alpha_1} \cap \dots \cap H_{\alpha_{a - 1}} \cap \\
H_{\alpha_{a+1}} \cap \dots \cap H_{\alpha_{a + x_1 - 1}} \cap \\ 
\vdots \\ 
H_{\alpha_{a + x_1 + \dots + x_{t-1} + 1}} \cap \dots \cap H_{\alpha_{a + x_1 + \dots + x_{t} - 1}} \cap \\
H_{\alpha_{a + x_1 + \dots + x_t + 1}} \cap \dots \cap H_{\alpha_{n - 1}} \cap H_{\beta}^b, \end{align*} where by convention we ignore $H_{\alpha_{a + x_1 + \dots + x_{i-1} + 1}} \cap H_{\alpha_{a+ x_1 + \dots  + x_i - 1}}$ if $x_i = 1$.
\end{defn}

Note that the special vertex in $\mathcal{C}$ lying on the hyperplanes $H_{\alpha_1},\dots,H_{\alpha_n-1},H_\beta$ is the facet $\langle 0| \ |n \rangle$, while the other special vertex in $\mathcal{C}$, which lies on $H_e$ instead of $H_\beta,$ is the facet $\langle n| \ |0 \rangle$

\vspace{0.15in}

\noindent \textbf{Example:} Let $\G = \Sp_4$. In Figure \ref{fig:Sp4facets} below, we label the facets in the alcove $\mathcal{C}$ determined by the simple roots $\Delta = \{ \alpha, \beta \}$. There are three vertices and three edges in addition to the interior of the alcove.

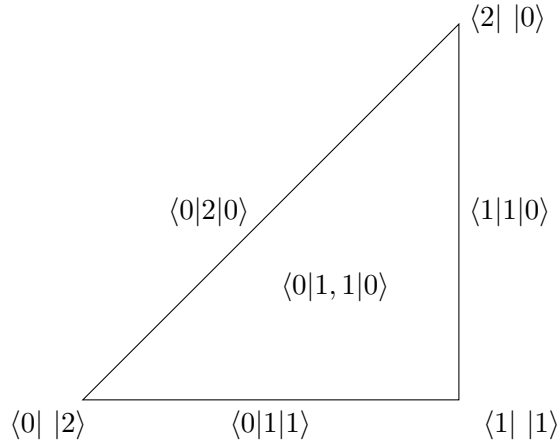
\begin{figure}[ht]
\centering
\begin{tikzpicture}
\draw (0,0) node[anchor=north]{\shortstack{$\langle 0 | \ | 2 \rangle$} \hphantom{hello}}
  -- (5,0) node[anchor=north]{$\hphantom{hellodolly} \langle 1 | \ | 1 \rangle$}
  -- (5,5) node[anchor=west]{\shortstack{$\langle 2 | \ | 0 \rangle$} $ \vphantom{P_{P_{P_{P_{P}}}}}$}
  -- cycle;
  \draw (2.5,0) node[anchor=north]{$\langle 0 | 1 | 1 \rangle$};
  \draw (1,2.5) node[anchor=west]{$\langle 0 | 2 | 0 \rangle$};
   \draw (5,2.5) node[anchor=west]{$\langle 1 | 1 | 0 \rangle$};
   \draw (2.5,1.5) node[anchor=west]{$\langle 0 | 1, 1 | 0 \rangle$};
\end{tikzpicture}
\caption{A labeling of the facets for Sp$_4$}
\label{fig:Sp4facets}
\end{figure}

\vspace{0.15in}

If $x \in \mathcal{B}(G),$ we let $G_x$ denote the parahoric subgroup of $G$ attached to $x$, and we let $G_x^+$ denote the prounipotent radical of $G_x$. Note that both $G_x$ and $G_x^+$ depend only on the facet $\textsf{F}$ to which $x$ belongs, so we may write $G_{\textsf{F}}$ and $G_{\textsf{F}}^+$. For a facet $\textsf{F}$ in $\mathcal{B}(G),$ the quotient $\textsf{G}_{\textsf{F}} := G_{\textsf{F}}/G_\textsf{F}^+$ is the group of $\mathbb{F}_q$-rational points of a reductive $\mathbb{F}_q$-group. The root system of this reductive group is the subdiagram of the extended Dynkin diagram associated to the facet. So in particular, the reductive quotient at the facet $\langle a | x_1, \dots, x_t | b \rangle$ has root system of type $$C_a \times A_{x_1 - 1} \times \dots \times A_{x_t - 1} \times C_b,$$ where again by convention we ignore terms with subscript 0. 

We now discuss the parameterization in \cite{debacker} for $\G$. Let $\mathcal{C}^S$ denote the set of $G$-conjugacy classes of maximal unramified tori. We will relate $\mathcal{C}^S$ to the set $I^m$ of pairs $(\textsf{F},\textsf{T}),$ where $\textsf{F}$ is a facet in $\mathcal{B}(G)$ and $\textsf{T}$ is a $\mathbb{F}_q$-minisotropic maximal torus in $\textsf{G}_{\textsf{F}}$. Given such a pair, DeBacker shows that $\textsf{T}$ can be lifted to a maximal unramified $F$-torus $T$, and he shows that a representative of each $G$-conjugacy class in $\mathcal{C}^S$ arises in this way. To attain a bijection, one needs to define an equivalence relation on $I^m$. 

Given a facet $\textsf{F}$ in an apartment $\mathcal{A}'$ of $\mathcal{B}(G)$, let $A(\mathcal{A}',\textsf{F})$ be the affine subspace of $\mathcal{A}'$ spanned by $\textsf{F}$. Then if $A(\mathcal{A}',\textsf{F}) = A(\mathcal{A}',\textsf{F}')$ for two facets $\textsf{F},\textsf{F}' \subseteq \mathcal{A}'$, we can identify the reductive quotients $\textsf{G}_{\textsf{F}}$ and $\textsf{G}_{\textsf{F}'}$ (see \cite[Lemma 3.5.1]{nilporb}). This identification then allows us to define an equivalence relation by saying $(\textsf{F},\textsf{T}) \sim (\textsf{F}',\textsf{T}')$ provided there is an apartment $\mathcal{A'}$ in $\mathcal{B}(G)$ and an element $g \in G$ so that 
\begin{itemize}
    \item $\emptyset \neq A(\mathcal{A}',\textsf{F}) = A(\mathcal{A}',g\textsf{F}')$ and
    \item $\textsf{T}$ is identified with $^g\textsf{T}'$ by the identification of $\textsf{G}_{\textsf{F}}$ with $\textsf{G}_{g\textsf{F}'}.$
\end{itemize}

Then this equivalence relation gives us a bijection 
$$I^m/\sim \\ \ \rightarrow \\ \\ \ \ \mathcal{C}^S.$$

Modeling upcoming work in \cite{adler}, we can refine this parameterization. First, note that since $\mathcal{C}$ is a fundamental alcove, every facet $\textsf{F}$ in $\mathcal{B}(G)$ is conjugate to at least one facet in $\mathcal{C}.$ Thus we may restrict to pairs $(\textsf{F},\textsf{T})$ with $\textsf{F} \subseteq \mathcal{C}$. Additionally, by \cite[Lemma 4.2.1]{debacker} or \cite{carter2}, the maximal $\mathbb{F}_q$-tori in the reductive quotient $\textsf{G}_{\textsf{F}}$ are parameterized by the conjugacy classes of the Weyl group $W_{\textsf{F}}$ associated to the image of $S \cap G_\textsf{F}$ in $\textsf{G}_{\textsf{F}}$, which we can and do the identify with a subgroup of the Weyl group $W$ of $S$ in $G$. Futhermore, the $\mathbb{F}_q$-minisotropic tori in the reductive quotient correspond to the elliptic conjugacy classes. Thus we can refine our parameterization to look at equivalence classes of pairs $(\textsf{F},w),$ where $\textsf{F}$ is a facet in $\mathcal{C}$ and $w$ is an elliptic element in the Weyl group $W_{\textsf{F}}.$

The elements in the Weyl group of $G$ can be identified with signed permutations of $\{ 1, \dots, n\}$. By ignoring the sign changes, each element $\tau$ of $W$ determines a permutation $\tau'$ of $\{ 1, \dots, n\},$ and this permutation can be written as a product of disjoint cycles. If $j$ is in a cycle of length $k$ in the cycle decomposition of $\tau'$, then $\tau^{k}(j) = \pm j$. If $\tau^{k}(j) = j$, then we say that the cycle is even, and if $\tau^{k}(j) = -j$, then we say that the cycle is odd. Mimicking the notation in \cite{carter1}, we write $C_k$ for an odd cycle of length $k$ and $A_{k - 1}$ for an even cycle of length $k$. In this notation, using Carter's classification of conjugacy classes in the Weyl group found in \cite{carter1}, given a facet $\textsf{F}$ in $\mathcal{C}$ of type $\langle a|x_1, \dots, x_t|b \rangle,$ the elliptic elements in $W_{\textsf{F}}$ are of the form $$(C_{a_1} \times \dots \times C_{a_k}) \times A_{x_1 - 1} \times \dots \times A_{x_t - 1} \times (C_{b_1} \times \dots \times C_{b_j}),$$ where $\sum_{i = 1}^k a_i = a $ and $\sum_{i = 1}^j b_i = b$.

For example, in $\Sp_{18}$, the facet $\langle 2 | 2, 1, 3 | 1\rangle$ is the facet vanishing on the hyperplanes $H_e, H_{\alpha_1}, H_{\alpha_3}, H_{\alpha_6}$, $H_{\alpha_7}$, and $H_\beta$. It has root system of type $C_2 \times A_1 \times A_2 \times C_1$, and the elliptic Weyl group elements are of type $C_2 \times A_1 \times A_2 \times C_1$ and $C_1 \times C_1 \times A_1 \times A_2 \times C_1$ in Carter's notation.

\vspace{0.15in}

\noindent \textbf{Example:} The conjugacy classes of maximal unramified tori in $\Sp_4$ are parameterized as in Figure \ref{fig:Sp4pairs} below. In this example, none of the facets in the alcove $\mathcal{C}$ determined by the simple roots $\Delta = \{ \alpha, \beta \}$ are equivalent, and there are nine pairs $(\textsf{F},w)$ up to equivalence.

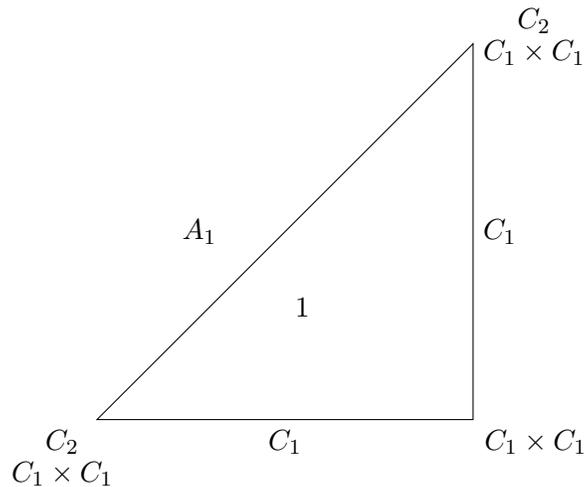
\begin{figure}[ht]
\centering
\begin{tikzpicture}
\draw (0,0) node[anchor=north]{\shortstack{$C_2$\\$C_1 \times C_1$} \hphantom{hello}}
  -- (5,0) node[anchor=north]{$\hphantom{hellodolly} C_1 \times C_1$}
  -- (5,5) node[anchor=west]{\shortstack{$C_2$\\$C_1 \times C_1$} $ \vphantom{P_{P_{P_{P_{P}}}}}$}
  -- cycle;
  \draw (2.5,0) node[anchor=north]{$C_1$};
  \draw (1,2.5) node[anchor=west]{$A_1$};
   \draw (5,2.5) node[anchor=west]{$C_1$};
   \draw (2.5,1.5) node[anchor=west]{$1$};
\end{tikzpicture}
\caption{A labeling of the pairs $(\textsf{F},w)$ for Sp$_4$}
\label{fig:Sp4pairs}
\end{figure}

\vspace{0.15in}

\section{The Waldspurger Parameterization for \textbf{G}}

We will first discuss Waldspurger's parameterization of conjugacy classes of regular semisimple elements in $\mathfrak{g},$ which will then allow us to discuss his parameterization of the conjugacy classes of maximal unramified tori.

\subsection{Regular Semisimple Elements in \textbf{G}}

We let $\mathfrak{g}_{reg}$ denote the set of regular semisimple elements of $\mathfrak{g}$. We are going to recall the description of the $G$-conjugacy classes in $\greg$ given in \cite{wald}. To start, we will consider the following objects:

\begin{enumerate}
    \item A finite set $I$
    \item For all $i \in I$, a finite extension $F_{i}^{\#}$ of $F$ and a \fnum{}-algebra $F_i$ which is 2-dimensional over \fnum.
    \item For all $i \in I,$ elements $a_i$ and $c_i$ in $F_i^{\times}.$
\end{enumerate}

For all $i \in I$, we let $\tau_i$ be the unique non-trivial automorphism of $F_i$ over $\fnum$. We assume that our choices above satisfy

\begin{enumerate}[(a)]
\item For all $i$, $a_i$ generates $F_i$ over $F$.
\item For all $i,j \in I$ with $i \neq j$, there does not exist an $F$-linear isomorphism from $F_i$ to $F_j$ which maps $a_i$ to $a_j$
\item For all $i \in I$, $\tau_i(a_i) = -a_i$ and $\tau_i(c_i) = -c_i$
\item $2n = \sum_{i \in I}[F_i:F]$
\end{enumerate}

Thus the choice of $a_i$ determines $F_i$ and $F_i^{\#}$. Write $W = \bigoplus _{i\in I} F_i$, and define a symplectic form $q_W$ on $W$ by $$q_W(\sum_{i\in I}w_i,\sum_{i\in I}w_i') = \sum_{i \in I} [F_i:F]^{-1} \text{trace}_{F_i/F}(\tau_i(w_i)w_i'c_i).$$

We let $X_W$ be the element of End$_F(W)$ defined by $X_W(\sum_{i \in I} w_i) = \sum_{i\in I} a_iw_i.$ Then if we fix an isomorphism from $(W,q_W)$ onto $(V,q_V)$, the element $X_W$ identifies an element $X \in \mathfrak{g}$ which is regular and semisimple. The orbit does not depend on the choice of isomorphism, and if we call it $\mathcal{O}(I,(a_i),(c_i)),$ then all orbits in \greg{} are of this form.

Now for all $i \in I$, let sgn$_{F_i/F_i^{\#}}$ be the quadratic character of $F_i^{\# \times}$ associated to the algebra $F_i$. Let $I^*$ be the set of $i \in I$ so that $F_i$ is a field, i.e. so that sgn$_{F_i/F_i^{\#}}$ is non-trivial. Then for two families $(I,(a_i),(c_i))$ and $(I',(a_i'),(c_i'))$ satisfying the above conditions, we have that the corresponding orbits $\mathcal{O}(I,(a_i),(c_i))$ and $\mathcal{O}(I',(a_i'),(c_i'))$ are equal if and only if
\begin{enumerate}
    \item There is a bijection $\phi: I \rightarrow I'$.
    \item For all $i \in I$, there is an $F$-linear isomorphism $\sigma_i: F_{\phi(i)}' \rightarrow F_i$ so that
    \begin{enumerate}[i.]
        \item For all $i \in I$, $\sigma_i(a_{\phi(i)}') = a_i$
        \item For all $i \in I$, sgn$_{F_i/F_i^{\#}}(c_i\sigma_i(c_{\phi(i)}')^{-1}) = 1$
    \end{enumerate}
\end{enumerate}

We also have that the two orbits are in the same stable conjugacy class, i.e. there are regular semisimple elements of $\mathfrak{g}$ in the respective orbits which are conjugate by an element of $\G{},$ if maps $\phi$ and $\sigma_i$ satisfying all but condition ii. above exist. A stable class $\mathcal{O}^{\text{st}}(I,(a_i),(c_i))$ thus splits into exactly $(\mathbb{Z}/2\mathbb{Z})^{I^*}$ $G$-conjugacy classes.

\subsection{Unramified Tori in \G}
Given a partition $\lambda = (x_1, \dots, x_k),$ let S$(\lambda) := \sum_{i=1}^k x_i$ denote the sum of the parts of $\lambda$. We always order the parts of $\lambda$ so that $x_1 \geq \dots \geq x_k,$ and we do not allow parts of $\lambda$ to be zero. Let $\theta_{\text{max}}(V)$ be the set of triples of partitions $(\mu_0, \mu', \mu'')$ so that $S(\mu_0) + S(\mu') + S(\mu'') = n$. Then we have a bijection $\mathcal{C}^S \rightarrow \theta_{\text{max}}(V)$ given as follows:

Let $T$ be a maximal unramified torus, and fix a regular semisimple element $X$ in $\mathfrak{t}$, the Lie algebra of $T$ in $\mathfrak{g}$. Choose $(I,(a_i),(c_i))$ so that $X$ is in the orbit $\mathcal{O}(I,(a_i),(c_i))$. For an integer $m \geq 1,$ let $F^{(m)}$ be the unique unramified extension of $F$ of degree $m$. Since $T$ is unramified, for all $i \in I$, there exists an integer $m(i)$ so that $F_i^{\#} = F^{(m(i))}$. Let $I'$, resp. $I''$, be the set of $i \in I^*$ so that the valuation $v_{F_i}(c_i)$ of $c_i$ in $F_i$ is even, resp. odd. We define the triple $(\mu_0,\mu',\mu'')$ by setting $\mu'$, resp. $\mu''$, to be the the partition which has the same parts as the family $(m(i))_{i \in I'}$, resp. $(m(i))_{i \in I''}$. Finally, we define $\mu_0$ to be the partition having the same non-zero terms as the family $(m(i))_{i \in I \setminus I^*}$. Thus we have defined an element of $\theta_{\text{max}}(V)$ from a given maximal unramified torus $T$. According to \cite{wald}, it does not depend on the choice of $X$ in $T$, and $^gX$ defines the same element of $\theta_{\text{max}}(V)$ for all $g \in G$. Thus the construction gives a bijection $$\mathcal{C}^s \rightarrow \theta_{\text{max}}(V)$$.

\section{Main Result}

We are now ready to compare the two parameterizations.

\begin{thm}
Consider $(\mu_0,\mu',\mu'') \in \theta_{\text{max}}(V)$, with $\mu_0 = (x_1, \dots, x_r)$, $\mu' = (y_1, \dots, y_s)$, $\mu'' = (z_1, \dots, z_t).$ Then we have that the corresponding $G$-conjugacy class in $\mathcal{C}^S$ corresponds to the equivalence class of $(\textsf{F},w)$ in the DeBacker parameterization,
where \textsf{F} is the facet $\langle \text{S}(\mu'')|x_1, \dots, x_r|\text{S}(\mu') \rangle$ and $w$ is in the conjugacy class of type $$(C_{z_1} \times \dots \times C_{z_t}) \times A_{x_1 - 1} \times \dots \times A_{x_r - 1} \times (C_{y_1} \times \dots \times C_{y_s}).$$
\end{thm}

In particular, the tori corresponding to the special vertex of $\mathcal{C}$ vanishing at the hyperplanes $H_{\alpha_1},\dots,H_{\alpha_{n-1}}$ and $H_\beta$ are those corresponding to triples of the form $( \emptyset, \mu', \emptyset )$. Similarly, the tori vanishing corresponding to the special vertex of $\mathcal{C}$ vanishing at the hyperplanes $H_{\alpha_1},\dots,H_{\alpha_{n-1}}$ and $H_e$ are those corresponding to triples of the form $( \emptyset, \emptyset, \mu'' ).$ The diagonal torus $S$ corresponds to the triple $( (1, \dots, 1), \emptyset, \emptyset ).$

\vspace{0.15in}

\noindent \textbf{Example:} In Figure \ref{fig:Sp4comp} below, for each pair $(\textsf{F},w)$ in Sp$_4$, we give the corresponding triple in the Waldspurger paramaterization.

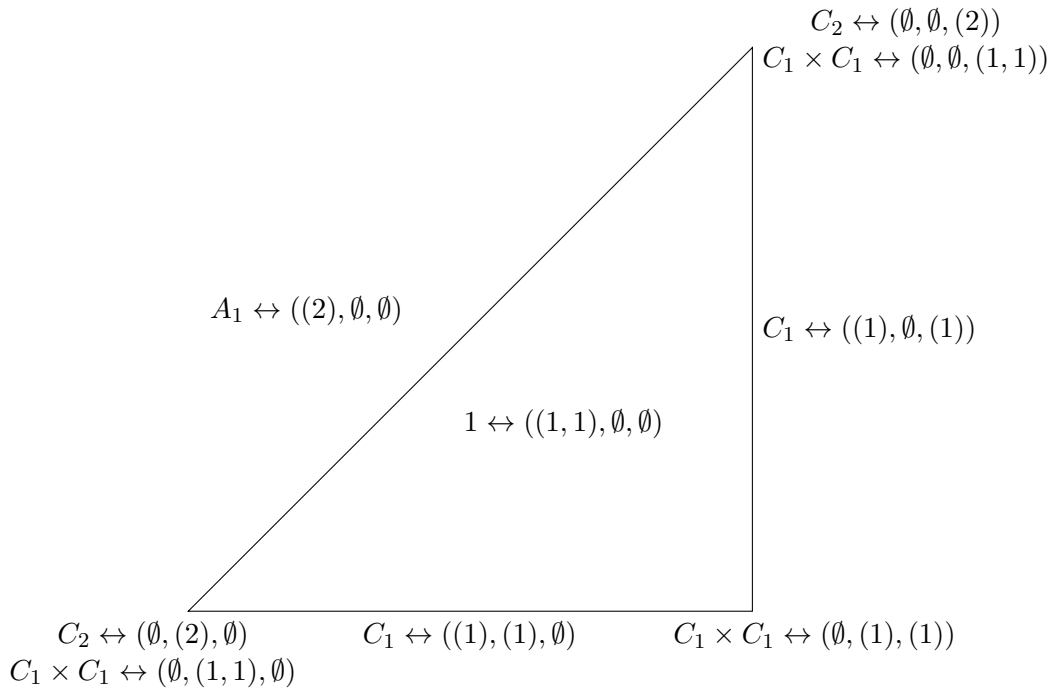
\begin{figure}[ht]
\centering
\begin{tikzpicture}
\draw (0,0) node[anchor=north]{\shortstack{$C_2 \leftrightarrow (\emptyset,(2),\emptyset)$\\$C_1 \times C_1 \leftrightarrow (\emptyset, (1,1), \emptyset) $} \hphantom{hello}}
  -- (7.5,0) node[anchor=north]{$\hphantom{hellodolly} C_1 \times C_1 \leftrightarrow (\emptyset,(1),(1))$}
  -- (7.5,7.5) node[anchor=west]{\shortstack{$C_2 \leftrightarrow (\emptyset,\emptyset,(2))$\\$C_1 \times C_1 \leftrightarrow (\emptyset,\emptyset,(1,1))$} $ \vphantom{P_{P_{P_{P_{P}}}}}$}
  -- cycle;
  \draw (3.75,0) node[anchor=north]{$C_1 \leftrightarrow ((1),(1),\emptyset)$};
  \draw (3,4) node[anchor=east]{$A_1 \leftrightarrow ((2),\emptyset,\emptyset)$};
   \draw (7.5,3.75) node[anchor=west]{$C_1 \leftrightarrow ((1),\emptyset,(1))$};
   \draw (5,2.5) node{$1 \leftrightarrow ((1,1),\emptyset,\emptyset)$};
\end{tikzpicture}
\caption{The Theorem for Sp$_4$}
\label{fig:Sp4comp}
\end{figure}

\vspace{0.15in}

We will prove this theorem throughout the rest of this section. We will deal with each of the partitions in a separate subsection before putting everything together at the end. For each partition, we will carefully construct the necessary field extensions, algebras, and generators $a_i$. We will then use our construction to produce an ordered symplectic basis and determine the structure of the matrix of the multiplication map with respect to our choice of ordered basis. We will pay particularly close attention to:

\begin{enumerate}
    \item The location of the non-zero entries of the matrix of multiplication by $a_i$, particularly in our analysis of $\mu_0$. We need to choose our basis of $W$ so that $X$ lies in the parahoric of our proposed facet $\textsf{F}$ and descends to a regular semisimple element in the reductive quotient. In particular, we will choose our ordered basis so that all of the non-zero entries of $X$ will lie on either the diagonal or in the root space of a root $\alpha$ so that $\textsf{F} \subseteq H_{\alpha}$.
    
    \item The way in which Fr acts on the eigenvalues of the matrix of multiplication by $a_i$. When we put everything together, we will use this information to determine the conjugacy class in $W$ associated to our torus.
\end{enumerate}

\subsection{The Partition $\mu_0$}

We begin with a lemma:

\begin{lemma}
The degree $m$ extension $\mathbb{F}_q^m$ of $\mathbb{F}_q$ contains an element $\eta$ so that $\mathbb{F}_q^m = \mathbb{F}_q(\eta^2)$. In other words, $\eta^2$ is a primitive element of the extension $\mathbb{F}_q^m/\mathbb{F}_q$.
\end{lemma}

\begin{proof}


We will treat the cases of $p=2$ and $p \neq 2$ separately. First, suppose $p \neq 2$, and choose an element $\eta \in \mathbb{F}_q^m$ so that $\eta$ is a cyclic generator of the multiplicative group of $\mathbb{F}_q^m.$ Then $\eta$ has order $q^m - 1,$ and since $p \neq 2$, $\eta^2$ has order $(q^m - 1)/2 = \dfrac{q - 1}{2}(1 + q + \dots + q^{m-1}).$ If $\eta^2$ is not a primitive element of the extension $\mathbb{F}_q^m/\mathbb{F}_q$, then $\eta^2$ must belong to some proper subextension, and so the order of $\eta^2$ is less or equal $q^{k}-1$ for some $k < m$. But $$q^k - 1 < q^k < 1 + q + \dots + q^{m-1} < \dfrac{q - 1}{2}(1 + q + \dots + q^{m-1})$$ since $q > p > 2$, and so $\eta^2$ is the desired primitive element of the extension $\mathbb{F}_q^m/\mathbb{F}_q.$

On the other hand, if $p = 2,$ then since 2 is then coprime to $q^m - 1$, $\eta^2$ is another cyclic generator of the multiplicative group of $\mathbb{F}_q^m$ for any cyclic generator $\eta$ of $\mathbb{F}_q^m,$ and so we are done.
\end{proof}

For the part $x_1$ of the partition $\mu_0$, we may thus assume that $F^{(x_1)} = F(\eta^2)$ for some $\eta \in \mathfrak{o}_{F^{(x_1)}}^{\times}$. Let $f(x)$ be the minimal polynomial of $\eta^2$ over $F$ so that $F^{(x_1)} = F[x]/(f(x)).$ Note that $F^{(x_1)} = F(\eta)$ as well, and so the minimal polynomial $g(x)$ of $\eta$ over $F$ also has degree $x_1$, which we will use shortly. Then we define the algebra $F_{x_1} := F^{(x_1)}[y]/(y^2 - \eta^2) \cong F[y]/(f(y^2))$. Then $F_{x_1}$ is 2-dimensional over $F^{(x_1)}$ and is not a field, meaning that it is the algebra corresponding to $x_1$ in the Waldspurger construction. We have that the non-trivial $F^{(x_1)}$-automorphism $\tau$ sends $y$ to $-y$. 

We set $c_1 = y$ so that our symplectic form on $F_{x_1}$, viewed as an $F$-algebra, is given by the pairing $\langle v_1, v_2 \rangle = \frac{1}{2x_1}\text{trace}_{F_{x_1}/F}(\tau(v_1)v_2y)$ for $v_1, v_2 \in F_{x_1}$. We set $a_1 = y$ so that our $F$-endomorphism $X_{x_1}$ on $F_{x_1}$ is given by multiplication by $y$.

We will now determine the eigenvalues of our multiplication map $X_{x_1}$, and we will determine how Fr acts on them. Again viewing $F_{x_1}$ as an $F$-vector space, we begin by fixing the ordered basis $\{1, y, \eta^2, \eta^2y, \eta^4, \dots, \eta^{2(x_1 - 1)}, \eta^{2(x_1 -1)}y\}$. Then the matrix of $X_{x_1}$ with respect to this basis is in rational canonical form, and so the characteristic polynomial of $X_{x_1}$ is $f(x^2)$. We then claim that $f(x^2) = \pm g(x)g(-x).$ To see this, first note that $\eta$ and $-\eta$ cannot both be roots of $g(x)$. If they were, we would have $\eta = \text{Fr}^k(-\eta)$ for some integer $k < x_1$, as every unramified extension of $F$ is cyclic. But then we would have that $\eta^2 = \text{Fr}^k(\eta^2),$ implying that the extension $\mathbb{F}_q(\eta^2)$ has degree less than $x_1$ and hence contradicting our choice of $\eta$. Thus the minimal polynomial of $-\eta$, which is $\pm g(-x)$, is not equal to $g(x)$, and so in particular $g(x)$ and $g(-x)$ have distinct roots. But $\eta$ and $-\eta$ are both roots of the monic polynomial $f(x^2)$, and so we must have that both $g(x)$ and $\pm g(-x)$ divide $f(x^2)$. Since the degrees match and the roots are distinct, we have our claim.

With this equality, we have that Fr permutes the roots of $g(x)$, which are elements of $\mathfrak{o}_{F^{(x_1)}}^\times$ having distinct images in the residue field, and the roots of $g(-x)$, which are the negatives of the roots of $g(x)$, cyclically, and so Fr acts on the roots of $f(x^2)$, which are the eigenvalues of $X_{x_1}$, via two $x_1$-cycles. 

We will now begin working towards a symplectic basis. To begin, we reorder our basis as $$\{1,\eta^2,\eta^4,\dots,\eta^{2(x_1-1)},y,\eta^2y,\dots,\eta^{2(x_1-1)}y\}.$$ Then with respect to the basis, $X_{x_1}$ is a block matrix of the form

\begin{center}
    $\begin{pmatrix}
    0 & A \\
    \Id & 0 
    
    \end{pmatrix}$
    
\end{center}

\noindent where each block is a $x_1 \times x_1$ matrix with non-zero entries in $\mathfrak{o}$. However, if we consider $F(\eta^2) = F^{(x_1)}$ as an $F$-vector space and let $Y$ be the linear transformation given by multiplication by $\eta^2$, we notice that $A$ is also the matrix of $Y$ with respect to the ordered basis $B = \{1, \eta^2, \eta^4, \dots, \eta^{2(x_1-1)} \}$. Thus the matrix $A$ has a natural square root $\sqrt{A}$ obtained by taking the matrix of the linear transformation on $F(\eta^2)$ given by multiplication by $\eta$ with respect to the basis $B$. We also see that $\sqrt{A}$ is invertible, with inverse given by multiplication by $\eta^{-1}$ in $F(\eta^2).$ Note that $\sqrt{A}$ and $\sqrt{A}^{-1}$ both have entries in $\mathfrak{o}$.





Now let $T$ be the block matrix

\begin{center}
    $\begin{pmatrix}
    \sqrt{A} & -\sqrt{A} \\
    \Id & \Id
    
    \end{pmatrix}$
\end{center}

\noindent and let $B_T$ be the transformed ordered basis $B_T = \{T \cdot 1, T \cdot \eta^2, \dots, T \cdot \eta^{2(x_1-1)}, T \cdot y, \dots, T \cdot \eta^{2(x_1-1)}y \}.$ Then the matrix of $X_{x_1}$ with respect to $B_T$ is given by

\begin{center}
    $\begin{pmatrix}
    \frac{1}{2}\sqrt{A}^{-1} & \frac{1}{2}\Id \\
    -\frac{1}{2}\sqrt{A}^{-1} & \frac{1}{2}\Id
    \end{pmatrix}$
    $\begin{pmatrix}
    0 & A \\
    \Id & 0 
    
    \end{pmatrix}$
    $\begin{pmatrix}
    \sqrt{A} & -\sqrt{A} \\
    \Id & \Id
    
    \end{pmatrix}$
    =
    $\begin{pmatrix}
    \sqrt{A} & 0 \\
    0 & -\sqrt{A}
    \end{pmatrix}$
    
\end{center}

From here, we finally construct a symplectic basis. First, we claim that $\langle T\cdot \eta^{2i}, T \cdot \eta^{2j} \rangle = 0 = \langle T\cdot \eta^{2i}y, T \cdot \eta^{2j}y \rangle$ for all $i,j$ in $\{0, \dots, x_1-1\}$. To see this, note that $$T \cdot \eta^{2i} = \eta \cdot \eta^{2i} + \eta^{2i}y = \eta^{2i + 1} + \eta^{2i}y$$ since $\sqrt{A}$ acts on $F^{(x_1)}$ by multiplication by $\eta$ by construction and $$T \cdot \eta^{2i} y = -\eta \cdot \eta^{2i} + \eta^{2i} y = -\eta^{2i + 1} + \eta^{2i} y.$$
Thus 
\begin{align*}
    \langle T\cdot \eta^{2i}, T \cdot \eta^{2j} \rangle &= \langle \eta^{2i + 1} + \eta^{2i}y, \eta^{2j + 1} + \eta^{2j}y \rangle = \frac{1}{2x_1} \text{trace}_{F_{x_1}/F} ((\eta^{2i+1} - \eta^{2i}y)(\eta^{2j + 1} + \eta^{2j}y)y)\\
    &= \frac{1}{2x_1} \text{trace}_{F_{x_1}/F} (y\eta^{2i + 2j + 2} + \eta^{2i + 2j + 3} - \eta^{2i + 2j + 3} - y\eta^{2i + 2j + 2})  \\
    &= 0.
\end{align*}
The computation for $\langle T\cdot \eta^{2i}, T \cdot \eta^{2j} \rangle$ is nearly identical.

Thus we can decompose the $F$-vector space $F_{x_1}$ into two Lagrangian subspaces, $x_1$-dimensional subspaces $F_{x_1}^+$ = span$_F\{T \cdot 1, T \cdot \eta^2, \dots, T \cdot \eta^{2(x_1-1)}\}$ and $F_{x_1}^-$ = span$_F\{T \cdot y, \dots, T \cdot \eta^{2(x_1-1)}y \}$ so that $\langle x, y \rangle = 0$ for all $x,y \in F_{x_1}^+,$ resp. $F_{x_1}^-$. Then for $v' \in F_{x_1}^-$, the map which sends $v \in F_{x_1}^+$ to $\langle v,v' \rangle$ gives a linear functional on $F_{x_1}^+$. Since the symplectic form is non-degenerate, we obtain an isomorphism $F_{x_1}^- \xrightarrow{\sim} (F_{x_1}^{+})^*,$ where $(F_{x_1}^{+})^*$ denotes the dual space of $F_{x_1}^+.$ We can then construct a dual basis $\{ \gamma_0^*, \dots, \gamma_{x_1-1}^* \}$ of $(F_{x_1}^+)^*$ so that $\gamma_i^*(T \cdot \eta^{2j}) = \begin{cases} 0 & i\neq j \\ 1 & i = j \end{cases}$, and letting $B'$ be the ordered basis $\{ T \cdot 1, T \cdot \eta^2, \dots, T \cdot \eta^{2(x_1-1)}, \gamma_{x_1-1}, \dots, \gamma_0\},$ where $\gamma_i$ is the inverse image of $\gamma_i^*$ under our isomorphism $F_{x_1}^- \xrightarrow{\sim} (F_{x_1}^{+})^*$, we have obtained our symplectic basis.

Note that the change of basis matrix from $B_T$ to $B'$ has the block-diagonal form $\begin{pmatrix} \Id & 0 \\ 0 & M \end{pmatrix}$ for some $x_1 \times x_1$ matrix $M$, and so the matrix of $X_{x_1}$ with respect to $B'$ also has block diagonal form 
\begin{center}
    $\begin{pmatrix}
    \sqrt{A} & 0 \\
    0 & N
    
    \end{pmatrix}$
\end{center}

\noindent where $N$ is the $x_1 \times x_1$ matrix given by conjugating $-\sqrt{A}^\intercal$ by the matrix
$\begin{pmatrix}
& & 1 \\
& \iddots & \\
1 & & 
\end{pmatrix}.$
Thus we note that the matrix of $X_{x_1}$ with respect to the ordered basis $B'$ still has entries in $\mathfrak{o}$. 


For future use, we let $\gamma_{j + 1}^{x_1,+} := T \cdot \beta^{2j}$ and $\gamma_{j+1}^{x_1,-} := \gamma_{j}$ for $j \in \{ 0, \dots, x_1 - 1 \}$ so that $B' = \{ \gamma_1^{x_1,+}, \dots, \gamma_{x_1}^{x_1,+}, \gamma_{x_1}^{x_1,-}, \dots, \gamma_{1}^{x_1,-}\}$. We also let $$\begin{pmatrix} M_{x_1}^{1,1} & 0 \\ 0 & M_{x_1}^{2,2} \end{pmatrix}$$ denote the matrix of $X_{x_1}$ with respect to this ordered basis.

\vspace{0.1in}

\noindent \textbf{Remark:} Note that we could have constructed the symplectic basis before applying the transformation $T$ to our original basis. However, the resulting matrix of $X_{x_1}$ would not have the block-diagonal form we have just obtained, which is essential to ensuring our resulting element will be a lift of a regular semisimple element in the reductive quotient of our eventual facet $\textsf{F}$. In particular, we want the non-zero entries of this piece to only show up in the root spaces in the subsystem spanned by the simple roots $\alpha_i$.

\vspace{0.1in}

Thus we are done with the entry $x_1$ in $\mu_0$. For the other $x_j \in \mu_0$, we can repeat this process. The only difference is if $x_{j_1} = x_{j_2}$ for $j_1 \neq j_2,$ we need to rescale our choice of $a_{j_2}$ by a representative of a non-trivial class in $\mathbb{F}_q^{\times} / (\mathbb{F}_q^{\times} \cap R_{2x_{j_1}})$, where $R_{2x_{j_1}}$ denotes the group of $2x_{j_1}$th roots of unity which belong to $\mathbb{F}_q^\times$. The resulting matrix of $X_{x_{j_2}}$ and its eigenvalues are then scaled by the same element so that the entries of $X_{x_{j_2}}$ still lie in $\mathfrak{o}$ and the eigenvalues of $X_{x_{j_2}}$ still lie in $\mathfrak{o}_{F^{(x_1)}}^\times$. Since $q > 2n$, we can scale so that for all $x_{j_m} = x_{j_1}$, each of the $a_{j_m}$ are multiplied by representatives of different classes. We need to do this so that condition (b) at the start of Section 4.1 is satisfied and so that the the eigenvalues of all the matrices $X_{x_{j_m}}$ are distinct elements of $\mathfrak{o}_{F_{\text{un}}}^\times$ with distinct images in $\overline{\mathbb{F}}_q^\times$.


\subsection{The Partition $\mu'$}

We choose a lift $\delta \in \mathfrak{o}_{F^{(y_1)}}^\times$ of an element of the degree $y_1$ extension $\mathbb{F}_q^{y_1}$ of $\mathbb{F}_q$ which generates the cyclic group $(\mathbb{F}_q^{y_1})^{\times}$. Then we have that the degree $y_1$ unramified extension $F^{(y_1)}$ is equal to $F[\delta]$. Furthermore, the image of $\delta$ cannot have a square root in $(\mathbb{F}_q^{y_1})^{\times}$ since $(\mathbb{F}_q^{y_1})^{\times}$ has even order, and so the degree 2 unramified extension $F_{y_1}$ of $F^{(y_1)}$ is generated by $\delta^{1/2}$ for some fixed square root of $\delta$. We let $\tau$ denote the non-trivial element of the Galois group of the extension $F_{y_1}/F^{(y_1)},$ so that $\tau(\delta^{1/2}) = -\delta^{1/2}$.

We choose $a_{y_1} = c_{y_1} = \delta^{1/2}$ in the Waldspurger construction. Thus our $F$-endomorphism $X_{y_1}$ of $F_{y_1}$ is given by multiplication by $\delta^{1/2}$, and our symplectic form on $F_{y_1}$ is given by the pairing $\langle v_1, v_2 \rangle = \frac{1}{2y_1}\text{trace}_{F_{y_1}/F}(\tau(v_1)v_2\delta^{1/2})$ for $v_1,v_2 \in F_{y_1}$. Let $f(x)$ be the minimal polynomial of $\delta$ over $F$. Then $f(x^2)$ is the minimal polynomial of $\delta^{1/2}$ over $x$, and so it also the characteristic polynomial of $X_{y_1}$. Since $f(x^2)$ is the minimal polynomial of an unramified extension of $F$, we know that the eigenvalues of $X_{y_1}$ are in $\mathfrak{o}_{F^{(y_1)}}^\times$ and have distinct images in $\mathbb{F}_q^{y_1}$. We also know that Fr permutes the $2y_1$ roots of $f(x^2)$ cyclically. Furthermore, for all roots $\gamma^{1/2}$ of $f(x^2),$ we must have that Fr$^{y_1}(\gamma^{1/2}) = -\gamma^{1/2}$ since Fr permutes the $y_1$ roots of $f(x)$ cyclically, meaning that Fr acts as an odd $y_1$-cycle on the roots of $f(x^2)$. 

 Before constructing our symplectic basis, we begin with the ordered basis $$\{ 1, \delta, \dots, \delta^{y_1 - 1}, \delta^{1/2}, \delta^{3/2}, \dots, \delta^{y_1 - 1/2} \}.$$ Then the matrix of our multiplication map with respect to this basis has the block form 
\begin{center}
    $\begin{pmatrix}
    0 & A \\
    \Id & 0 
    
    \end{pmatrix}$
\end{center}

\noindent where each block is a $y_1 \times y_1$ matrix. If we let $$F^+_{y_1} = \text{span}_F\{ 1, \delta, \delta^2, \dots, \delta^{y_1 - 1} \} \hspace{0.5in} \text{and} \hspace{0.5in} F^-_{y_1} = \text{span}_F\{ \delta^{1/2}, \delta^{3/2}, \dots, \delta^{y_1 - 1/2} \},$$ then note that any two elements of $F^+_{y_1}$ (resp. $F^-_{y_1}$) are mutually orthogonal with respect to our symplectic form. For $v'$ in $F^-_{y_1}$, the map which sends $v \in F^+_{y_1}$ to $\langle v, v' \rangle$ gives a linear functional on $F_{y_1}^+,$ and since the symplectic form is non-degenerate, we have an isomorphism $F_{y_1}^- \xrightarrow{\sim} (F^+_{y_1})^*,$ where $(F^+_{y_1})^*$ denotes the dual basis of $F_{y_1}^+$. Constructing a dual basis $\{ \gamma_0^*, \dots, \gamma_{y_1 - 1}^* \}$ of $(F_{y_1}^+)^*$ so that $\gamma_i^*(\delta^j) = \begin{cases} 0 & i \neq j \\ 1 & i = j \end{cases}$, we let $B' = \{ 1, \dots, \delta^{y_1 - 1}, \gamma_{y_1 - 1}, \dots, \gamma_0 \},$ where $\gamma_i$ is the inverse image of $\gamma_i^*$ under our isomorphism $F_{y_1}^- \xrightarrow{\sim} (F_{y_1}^+)^*.$ For future use, we let $\gamma_{j+1}^{y_1,+} = \delta^{j}$ and $\gamma_{j+1}^{y_1,-} = \gamma_j$ for all $j \in \{ 0, \dots, y_1 - 1 \}$ so that $B' = \{ \gamma_1^{y_1,+}, \dots, \gamma_{y_1}^{y_1,+}, \gamma_{y_1}^{y_1,-}, \dots, \gamma_{1}^{y_1,-} \}.$

Note that the change of basis matrix from our initial basis to the ordered basis $B'$ has the block-diagonal form $\begin{pmatrix} \Id & 0 \\ 0 & M \end{pmatrix}$ for some $y_1 \times y_1$ matrix $M$, and so the matrix of our multiplication map with respect to $B'$ has the block form

\begin{center}
    $\begin{pmatrix}
    0 & M_{y_1}^{1,2} \\
    M_{y_1}^{2,1} & 0 
    \end{pmatrix}.$
\end{center}

\noindent This is a matrix with entries in $\mathfrak{o}.$

For the other $y_j \in \mu',$ we can repeat this process. The only change we need to make is that if $y_{j_1} = y_{j_2},$ for $j_1 \neq j_2,$ we need to rescale our choice of $a_{y_{j_2}}$ by a representative of a non-trivial class in $\mathbb{F}_q^{\times} / (\mathbb{F}_q^{\times} \cap R_{2y_{j_1}})$, where $R_{2y_{j_1}}$ denotes the group of $2y_{j_1}$ roots of unity. The resulting matrix of $X_{y_{j_2}}$ and its eigenvalues are then scaled by the same element of $\mathbb{F}_q^{\times}$ so that the entries of $X_{y_{j_2}}$ still lie in $\mathfrak{o}$ and the eigenvalues of $X_{y_{j_2}}$ still lie in $\mathfrak{o}_{F^{(y_1)}}^\times$. Since $q > 2n$, we can scale so that for all $y_{j_m} = y_{j_1}$, all of the $a_{j_m}$ are multiplied by representatives of different classes. As with $\mu_0,$ we do this to ensure condition (b) at the start of section 4.1 is satisfied. 

\subsection{The Partition $\mu''$}

For $z_i \in \mu''$, we choose $\delta$ as in the previous section so that $F^{(z_i)} = F(\delta)$ and $F_{z_i} = F(\delta^{1/2}).$ The primary difference in the construction from the previous section is that we set $c_{z_i} = \varpi \delta^{1/2}$, so that the symplectic form is given by $\langle v_1, v_2 \rangle = \frac{1}{2z_i}\text{tr}_{F_{z_i}/F}(\tau(v_1)v_2\varpi \delta^{1/2})$ for $v_1, v_2 \in F_{z_i}.$  We take $a_{z_i} = d_{z_i} \cdot \delta^{1/2}$, where $d_{z_i}$ is an element of $\mathbb{F}_q$ chosen so that for all $z_i = z_j$, $d_{z_i}$ and $d_{z_j}$ are distinct representatives of $\mathbb{F}_q^{\times} / (\mathbb{F}_q^{\times} \cap R_{2z_i})$, where $R_{2z_{j_1}}$ denotes multiplicative group of $2z_i$th roots of unity in $\mathbb{F}_q^{\times}$, and so that for all $z_i = y_k$ for $y_k \in \mu'$, $a_{y_k}$ was not scaled by an element in the same class as $d_{z_i}$ in the previous section. Then our $F$-endomorphism is multiplication by $d_{z_i} \cdot \delta^{1/2}$. We take our initial basis to be $B = \{ 1, \delta, \dots, \delta^{z_i - 1}, \frac{\delta^{1/2}}{\varpi}, \frac{\delta^{3/2}}{\varpi}, \dots, \frac{\delta^{z_i - 1/2}}{\varpi} \}.$ We then construct the change of basis matrix that we would have constructed for the basis $\{ 1, \delta, \dots, \delta^{z_i - 1}, \delta^{1/2}, \delta^{3/2}, \dots, \delta^{z_i - 1/2} \}$ in the $\mu'$ case, and apply it to $B$ to obtain an ordered basis $B' = \{ \gamma_1^{z_i,+}, \dots, \gamma_{z_i}^{z_i,+}, \gamma_{z_i}^{z_i,-}, \dots, \gamma_{1}^{z_i,-} \}$. Then we again have a symplectic basis, and the matrix of the multiplication map $X_{z_i}$ with respect to $B'$ has the block form

\begin{center}
    $\begin{pmatrix}
    0 & M_{z_i}^{1,2} \\
    M_{z_i}^{2,1} & 0 
    \end{pmatrix}.$
\end{center}

\noindent The $z_i \times z_i$ matrix $M_{z_i}^{1,2}$ (resp. $M_{z_i}^{2,1})$ has non-zero entries in $\mathfrak{p}^{-1}$ (resp. $\mathfrak{p}$), and we again see that the eigenvalues of $X_{z_i}$ are elements of $\mathfrak{o}_{F^{(z_i)}}^\times$ with distinct images in $\mathbb{F}_q^{z_i}.$

\subsection{Completing the Proof}

We are now ready to put everything together. Let $$W = \bigoplus_{x_i \in \mu_0} F_{x_i} \oplus \bigoplus_{y_j \in \mu'} F_{y_j} \oplus \bigoplus_{z_k \in \mu''} F_{z_k}.$$

Then $W$ is the space associated to the triple $(\mu_0, \mu', \mu'')$ in the Waldspurger construction. The symplectic form and $F$-endomorphism determining our regular semisimple element are the respective sums of the symplectic forms and $F$-endomorphisms of each subspace, as described in section 4. 

For $x_i \in \mu_0$, set $B_{x_i}^+ = \{ \gamma_1^{x_i,+},\dots, \gamma_{x_i}^{x_i,+} \}$ and $B_{x_i}^- = \{ \gamma_{x_i}^{x_i,-}, \dots, \gamma_{1}^{x_i,-} \} $. Define $B_{y_j}^\pm$ and $B_{z_k}^\pm$ analogously. Fix the ordered basis $$B_W = B_{z_1}^+ \cup \dots \cup B_{z_t}^+ \cup B_{x_1}^+ \cup \dots \cup B_{x_r}^+ \cup B_{y_1}^+ \cup \dots \cup B_{y_s}^+ \cup B_{y_s}^- \cup \dots \cup B_{y_1}^- \cup B_{x_r}^- \cup \dots \cup B_{x_1}^- \cup B_{z_t}^- \cup \dots \cup B_{z_1}^-.$$

By our work in the previous subsections, we know that the matrix of our symplectic form on $W$ has the form

\begin{center}
$\begin{pmatrix}
& & & & & 1 \\
& & & & \iddots & \\
& & & 1 & & \\
& & -1 & & & \\
& \iddots & & & & \\
-1 & & & & & \\
\end{pmatrix}$
\end{center}

\noindent with respect to the ordered basis $B_W$ so that we have an obvious isomorphism from $W$ into $V$. With respect to $B_W$, we have that our $F$-endomorphism (which again is a regular semisimple element of our unramified torus), has the block form

\begin{center}
$\begin{pmatrix}
M^{1,1} & M^{1,2} \\
M^{2,1} & M^{2,2}
\end{pmatrix}$,
\end{center}

\noindent where each block has size $n \times n$. These blocks $M^{i,j}$ are in turn block matrices. The matrix $M^{1,1}$ is block diagonal of the form

\begin{center}
    $\begin{pmatrix}
    0 & & & & & & & & \\
    & \ddots & & & & & & & \\
    & & 0 & & & & & & \\
    & & & M_{x_1}^{1,1} & & & & & \\
    & & & & \ddots & & & & \\
    & & & & & M_{x_r}^{1,1} & & & \\
    & & & & & & 0 & & \\
    & & & & & & & \ddots & \\
    & & & & & & & & 0 \\
    \end{pmatrix}$,
\end{center}

\noindent where the blocks have size $z_1 \times z_1, \dots, z_t \times z_t, x_1 \times x_1, \dots, x_r \times x_r, y_1 \times y_1, \dots, y_s \times y_s$ respectively. Similarly, the block $M^{2,2}$ is block diagonal of the form

\begin{center}
    $\begin{pmatrix}
    0 & & & & & & & & \\
    & \ddots & & & & & & & \\
    & & 0 & & & & & & \\
    & & & M_{x_r}^{2,2} & & & & & \\
    & & & & \ddots & & & & \\
    & & & & & M_{x_1}^{2,2} & & & \\
    & & & & & & 0 & & \\
    & & & & & & & \ddots & \\
    & & & & & & & & 0
    \end{pmatrix}$,
\end{center}

\noindent where the listed blocks have size $y_s \times y_s, \dots, y_1 \times y_1, x_r \times x_r, \dots, x_1 \times x_1, z_t \times z_t, \dots, z_1 \times z_1$. The matrices $M^{1,2}$ and $M^{2,1}$ are block anti-diagonal. $M^{1,2}$ has the form

\begin{center}
    $\begin{pmatrix}
    & & & & & & & & M_{z_1}^{1,2} \\
    & & & & & & & \iddots & \\
    & & & & & & M_{z_t}^{1,2} & & \\
    & & & & & 0 & & & \\
    & & & & \iddots & & & & \\
    & & & 0 & & & & & \\
    & & M_{y_1}^{1,2} & & & & & & \\
    & \iddots & & & & & & & \\
    M_{y_s}^{1,2} & & & & & & & &

    \end{pmatrix}$,
\end{center}

\noindent where the listed blocks have size, from top to bottom, $z_1 \times z_1, \dots, z_t \times z_t, x_1 \times x_1, \dots, x_r \times x_r, y_1 \times y_1, \dots, y_s \times y_s$, and $M^{2,1}$ has the form

\begin{center}
    $\begin{pmatrix}
    & & & & & & & & M_{y_s}^{2,1} \\
    & & & & & & & \iddots & \\
    & & & & & & M_{y_1}^{2,1} & & \\
    & & & & & 0 & & & \\
    & & & & \iddots & & & & \\
    & & & 0 & & & & & \\
    & & M_{z_t}^{2,1} & & & & & & \\
    & \iddots & & & & & & & \\
    M_{z_1}^{2,1} & & & & & & & &

    \end{pmatrix}$,
\end{center}

\noindent where the listed blocks have size, from top to bottom, $y_s \times y_s, \dots, y_1 \times y_1, x_r \times x_r, \dots, x_1 \times x_1, z_t \times z_t, \dots, z_1 \times z_1$. Henceforth, we identify our $F$-endomorphism of $W$ with the matrix we have constructed above, and we will write $X_W$ for both.

Thus $X_W$ is the regular semisimple element for our torus arising from the Waldspurger construction. We now determine a pair $(\textsf{F},w)$ giving our torus in the DeBacker parameterization. We begin with $w$. To compute $w$, we first note that there is a diagonal matrix $s$ in the Lie algebra of $S$ and an element $g \in \textbf{G}(F_{\text{un}})$ so that $X_w ={} ^gs$. Then since $X_W$ is $F$-rational, we have that $n = g^{-1}\text{Fr}(g)$ is an element of $N_G(S)$ sending $\text{Fr}(s)$ to $s$, and we know from \cite{adler} that $w$ is the image of this element in the Weyl group. But by linear algebra, note that the columns of $g$ must be an eigenbasis for the endomorphism $X_W$. Since $X_W$ is a regular semisimple element in the Lie algebra of Sp$_{2n}$ its eigenvalues are distinct, and so $g^{-1}\text{Fr}(g)$ is a permutation matrix, where the cycle type of the permutation is determined by how Fr permutes the eigenvalues of $X_W$. But by our previous analysis, each $x_i \in \mu_0$ gives us an even cycle of type $A_{x_i - 1}$ in the Carter notation, while each $y_j \in \mu'$ and $z_k \in \mu''$ gives us an odd cycle of type $C_{y_j}$ or $C_{z_k}$ respectively. Thus the Weyl group element $w$ associated to $X_W$ is of type 

$$(C_{z_1} \times \dots \times C_{z_t}) \times A_{x_1 - 1} \times \dots \times A_{x_r - 1} \times (C_{y_1} \times \dots \times C_{y_s})$$

\noindent in the Carter notation.

To conclude the proof, we need to show that the image of $X_W$ in Lie($G_{\textsf{F}}$)/Lie($G_{\textsf{F}}^+$) is still a regular semisimple element, where $\textsf{F}$ denotes the facet in the fundamental alcove $\mathcal{C}$ of type $\langle \text{S}(\mu'')|x_1, \dots, x_t|\text{S}(\mu') \rangle$. To do this, we claim that it suffices to check

\begin{itemize}
    \item The eigenvalues of $X_W$ lie in $\mathfrak{o}_{F_{\text{un}}}^\times$ and are distinct in the reductive quotient.
    
    \item Every non-zero entry in the root space of a root $\alpha$ in the subsystem $\Phi'$ spanned by $\{ -e, \alpha_1, \dots, \alpha_{\text{S}(\mu'') - 1} \}$ lies in $\mathfrak{o}$ if the coefficient of $-e$ in $\alpha$ is 0 with respect to the simple system $\{ -e, \alpha_1, \dots, \alpha_{\text{S}(\mu'') - 1} \}$ of $\Phi'$, lies in $\mathfrak{p}$ if the coefficient of $-e$ is positive, or lies in $\mathfrak{p}^{-1}$ if the coefficient of $-e$ is negative.
    
    \item Every non-zero entry in the root space of a root $\alpha$ in the subsystem spanned by the simple roots $\alpha_j$ corresponding to the vertices in the subdiagram of the extended Dynkin diagram associated to a part $x_i \in \mu_0$ lies in $\mathfrak{o}$.
    
    \item Every non-zero entry in the root space of a root $\alpha$ in the subsystem spanned by the simple roots $\beta$ and $\alpha_{n - \text{S}(\mu') + 1}, \dots, \alpha_{n-1}$ is in $\mathfrak{o}$.
    
    \item Every non-zero diagonal entry is in $\mathfrak{o}$. 
    
    \item Every other entry is zero.
\end{itemize}

The last five conditions ensure that the element $X_W$ lies in the Lie algebra of the parahoric at our facet, and the first ensures that the image $\textsf{X}_W$ of $X_W$ in the associated reductive quotient is a regular semisimple element of Lie$(\textsf{G}_\textsf{F})$. Note that the centralizer of $\textsf{X}_W$, call it $\textsf{T}$, in the quotient $G_{\textsf{F}}/G_{\textsf{F}}^+$ is a maximal $\mathbb{F}_q$-torus that corresponds to the image of $C_G(X_W) \cap G_{\textsf{F}}$ in $G_{\textsf{F}}/G_{\textsf{F}}^+$. By \cite{debacker}, the maximal unramified torus $C_G(X_W)$ is a lift of $\textsf{T}$.

We now check that $X_W$ does in fact satisfy the six conditions. The first condition follows from our choice of the elements $a_i$ in the preceding sections for each partition. In particular, we showed that the eigenvalues associated to each part are distinct elements of $\mathfrak{o}_{F_{\text{un}}}^\times$, and since they are the roots of the minimal polynomial of an unramified extension, they have distinct image in the residue field. Furthermore, by scaling at the end of each subsection whenever one of our partitions contains parts that are equal, we ensured that eigenvalues of $X_W$ in the direct sum are also distinct elements of $\mathfrak{o}_{F_{\text{un}}}^\times$ which still have distinct image in $\overline{\mathbb{F}}_q^\times$.

We can see that all of the remaining conditions hold from the block form of our matrix $X_W$. In the first S$(\mu'')$ columns and last S$(\mu'')$ columns of $X_W$, all non-zero entries are in $M^{2,1}$ and $M^{1,2}$ respectively. In particularly, they lie in a block of the form $M_{z_i}^{2,1}$ or $M_{z_i}^{1,2}$. The entries in the blocks $M_{z_i}^{2,1}$ correspond to root spaces of roots in the subsystem spanned by $\{ -e, \alpha_1, \dots, \alpha_{\text{S}(\mu'') - 1} \}$ so that the coefficient of $-e$ is positive. By our construction in the previous section, we know that the non-zero entries of $M_{z_i}^{2,1}$ lie in $\mathfrak{p}$. Similarly, the entries in the blocks $M_{z_i}^{1,2}$ correspond to root spaces of roots so that the coefficient of $-e$ is negative, and again by our construction in the previous section, we know that the non-zero entries of $B_{z_i}^{1,2}$ lie in $\mathfrak{p}^{-1}.$ 

For $x_i \in \mu_0$, we consider the columns $S(\mu'') + x_1 + \dots + x_{i - 1} + 1, \dots, S(\mu'') + x_1 + \dots + x_{i - 1} + x_i$ and $2n - (S(\mu'') + x_1 + \dots + x_{i}) + 1, \dots, 2n - (S(\mu'') + x_1 + \dots + x_{i-1})$. Then the only non-zero entries are in the matrices $M_{x_i}^{1,1}$ and $M_{x_i}^{2,2}$. The non-zero entries of these matrices are in $\mathfrak{o}$. They contain the diagonal entries of $X_W$ lying in the columns, and they also contain the root spaces of the roots in the span of $\alpha_{S(\mu'') + x_1 + \dots + x_{i - 1} + 1}, \dots, \alpha_{S(\mu'') + x_1 + \dots x_i - 1}$, which are the simple roots corresponding to the vertices in the subdiagram of the extended Dynkin diagram associated to $x_i$.

Finally, in the middle $2 \cdot$ S$(\mu')$ columns of $X_W$, all non-zero entries are in the blocks of the form $M_{y_i}^{2,1}$ or $M_{y_i}^{1,2}.$ But the non-zero entries of these blocks are in $\mathfrak{o}$, and the entries of the blocks are root spaces of roots in the subsystem spanned by $\beta$ and $\alpha_{n - \text{S}(\mu')}, \dots, \alpha_{n-1}$.

Consequently, $X_W$ defines the torus associated to the pair $(\textsf{F},w),$ and so we are done.

\section{Inverse Map}

The construction in the previous section produces one pair $(\textsf{F},w)$ in the equivalence class associated to our maximal unramified torus. It is natural to ask whether we can obtain other pairs $(\textsf{F}',w')$ in our equivalence
class so that we may construct an inverse map. To answer this, note that our construction only produces facets of the form $\langle a|x_1, \dots, x_r|b \rangle$ for $x_1 \geq \dots \geq x_r$. However, in our construction of $W$ in the previous section, note that if $\sigma$ is a permutation of the parts of $\mu_0$, then we can permute the sub-bases $B_{x_i}^+$ by $\sigma$ when choosing our ordered basis $B_W$, as long as we also permute the $B_{x_i}^-$ analogously. Then with respect to this permuted ordered basis, the matrix of $X_W$ will still be of type $w$, and it will define an element in the same $G$-conjugacy class. However, now it will be regular semisimple in the reductive quotient of the facet of type $\langle \text{S}(\mu'')| \sigma(x_1), \dots, \sigma(x_r)| \text{S}(\mu') \rangle$.

Consequently, given a pair $(\textsf{F},w),$ where $\textsf{F}$ is of the form $\langle a | x_1, \dots, x_r | b \rangle$ and $w \in W_{\textsf{F}}$ is elliptic of type $(C_{a_1} \times \dots \times C_{a_t}) \times A_{x_1 - 1} \times \dots \times A_{x_r - 1} \times (C_{b_1} \times \dots \times C_{b_s})$, then we can define $\mu_0$ to be the unique partition containing the parts $x_1, \dots, x_r$, $\mu'$ to be the partition containing the parts $b_1, \dots, b_s$, and $\mu''$ to be the partition containing the parts $a_1, \dots, a_t$. Then the $F$-conjugacy class associated to the triple $(\mu_0, \mu', \mu'')$ is the $F$-conjugacy class associated to $(\textsf{F},w)$.

\end{document}